\documentclass[a4paper,reqno,12pt]{amsart}
\usepackage[T2A]{fontenc}
\usepackage[cp1251]{inputenc}
\usepackage{amsfonts, amssymb, amsmath}

\theoremstyle{plain}
\newtheorem{corollary}{Corollary}
\newtheorem{theorem}{Theorem}

\theoremstyle{definition}
\newtheorem{example}{Example}
\newtheorem{definition}{Definition}

\renewcommand{\le}{\leqslant}
\renewcommand{\ge}{\geqslant}
\newcommand{\Le}{\le}
\newcommand{\Ge}{\ge}
\DeclareMathOperator {\ord}{ord}
\newcommand{\N}{\mathbb{N}}

\newcommand{\Z}{\mathbb{Z}}

\def\deg{\operatorname{deg}}
\def\ord{\operatorname{ord}}
\def\C{\mathbb C}

\newcommand{\K}{\mathbb{K}}
\def\F{\mathcal {F}}
\def\G{\mathcal {G}}

\title[]{
Some Properties of Coefficients
Kolchin Dimension Polynomial.
}
\author{M.V.Kondratieva}
\thanks{
This research has been supported by the Interdisciplinary Scientific and Educational School of Moscow University «Brain, Cognitive Systems, Artificial Intelligence»
}
\address{
Moscow State University,  
Department of Mechanics and Mathematics, \\Leninskie Gory, Moscow, 
Russia,  119992.}
\email{marina.kondratieva@math.msu.ru}

\begin{document}

\begin{abstract}
The article presents a formula expressing
Macaulay 
constants of a numerical polynomial
through its  minimizing coefficients.
From this, we have that
Macaulay constants of Kolchin dimension polynomials
do not decrease.

For the minimal differential
dimension polynomial $\omega_{\G/\F}$ 
(this concept was introduced by W.Sitt in
\cite{Sit}) we will prove a criterion for Macaulay constants 
to be equal.
In this case, as the example (\ref{ex2}) shows, there are no 
bounds
from above to the Macaulay constants of the polynomial 
$\omega_{\xi/\F}$ for $\G=\F\left<\xi\right>$.

\end{abstract}

\maketitle
\thispagestyle{empty}

\section{Introduction}
One of the basic objects of stude in differential algebra is the
differential dimension polynomial introduced by E.Kolchin~\cite{Kolchin}. 
This is an analogue of dimension in algebraic  geometry,
and its role is similar to that of the Hilbert polynomial
in commutative algebra.

The dimension polynomial itself is not an 
invariant of the differential field,
but contains some invariants, such as the degree and leading coefficient.
In this article, we are considering
another little studied invariant,
introduced by W.Sit~\cite{Sit}, namely a minimal 
dimension polynomial.

Note that since the 80s of the last century, interest in computer algebra has increased.
Speaking of history,
one of the first scientists whose work laid to the foundation of
constructive theory of ideals in the ring of polynomials,
is F. Macaulay (20s of the 20th century).
He proved, in particular, 
the criterion for a numerical  function to be equal to
Hilbert function of a homogeneous ideal in the ring of commutative 
polynomials.
In 1990 T.Dub\'e used in the proof of the bound of the 
degree of elements
{G}r{\"o}bner
 basis numbers named in honor of Macaulay
as Macaulay constants.

The Hilbert function for sufficiently large values becomes a polynomial,
which is discussed in this article.
Moreover, the set of all Hilbert polynomials
 of homogeneous ideals coincides with the set of Kolchin 
dimension polynomials.
In \cite{KLMP} (proposition 2.4.10) a criterion for the equality of a
numerical polynomial to some Hilbert polynomial is proved, 
which does not use Macaulay's results.
In this paper we present this result in terms of Macaulay constants.

\section{Preliminary facts.}

One can find 
basic concepts and facts  in~\cite{Kolchin, KLMP}.

Denote   the set  of integers by $\Z$, 
non-negative integers by $\N_0$.
For $(i_1,\dots,i_m)\in\N_0^m$, the order of  $e$ is defined by 
$\ord e=\sum_{k=0}^mj_k$. 
We also denote by
$\binom{n}{k}$ binomial coefficients, $\frac{n!}{k!(n-k)!}.$
Note that any numerical (i.e., taking integer values at integer points) 
polynomial $v(s)$ can be written as $v(s)=\sum_{i=0}^da_i\binom{s+i}{i}$,
where $a_i\in\Z$.  We will call the numbers $(a_d,\dots,a_0)$ 
{\bf standard coefficients } of the polynomial $v(s)$.
The following definition first appeared in \cite{Kon}).

\begin{definition}\label{min}(see \cite{KLMP}, definition 2.4.9).
Let $\omega=\omega(s)$  
be a numerical polynomial of degree
$d$ in variable 
$s$. The sequence of  {\bf minimizing coefficients}  
of polynomial
$\omega$  is the vector
$b(\omega)=(b_d,\dots ,b_0)\in\Z^{d+1}$, 
defined
inductively on $d$ as follows. 
If $d=0$ (i.e., $\omega$ is a constant), then we set
$b(\omega)=(\omega)$. Let $d>0$ and $(a_d,\dots,a_0)$ are standard coefficients of polynomial $\omega$;
i.e.,
  $\omega(t)=\sum^d_{i=0}a_i\binom{t+i}{i}$.
 Denote by  $v(s) = \omega(s+a_d) - \binom{s+d+1+a_d}{d+1} + \binom{s+d+1}{d+1}$. Since
$\deg v<d$, one can calculate the sequence of minimizing coefficients
$b(v)=(b_k,\dots ,b_0)$ $(0\leq k<d)$ of polynomial $v(s)$. 
Now,   let     $b(\omega)=(a_d,0,\dots,0,b_k,\dots ,b_0)\in\Z^{d+1}$.
\end{definition}

Now  we define the Kolchin dimension polynomial
of a subset $E\subset \N_0^m$.
Regard the following partial order on $\N_0^m$: the relation
$(i_1,\dots, i_m)\le(j_1,\dots, j_m)$ is equivalent to
$i_k\le j_k$ for all $k=1,\dots, m$.
We  consider a function $\omega_E(s)$, that   in a point $s$ 
equals Card $V_E(s)$, where $V_E(s)$ is the set
of points $x\in\N_0^m$ such that $\ord x\le s$ 
and for every $e\in E$ the condition $e\le x$ isn't true.
Then (see for example, \cite{Kolchin}, p.115, or \cite{KLMP}, 
theorem 5.4.1 ) function $\omega_E(s)$ for all sufficiently large 
$s$ is a numerical polynomial. We call this polynomial
the {\bf Kolchin dimension polynomial} of a subset $E$and denote
$\omega_E(s)$.
\begin{definition}\label{3.2.1} An operator $\partial$ 
on a
commutative ring
$\K$ with unit is called
{\bf  a derivation 
} 
if it is linear
$\partial(a+b) =\partial(a)+\partial(b)$ and the Leibniz's rule  
$\partial(ab) =\partial
(a)b+a\partial(b)$ holds for all elements $a, b\in \K$.

{\bf A differential ring} (or $\Delta$-ring) is a 
ring $\K$ endowed with a set of derivations
$\Delta=\{\partial_1,\ldots,\partial_m\}$
which commute pairwise. 

Let 
$$\Theta =\Theta(\Delta) = \left
\{\partial_1^{i_1}\cdot\ldots\cdot\partial_m^{i_m}\:\big|\: i_j\Ge 0,\ 1\Le j\Le m\right\}$$
and
$\theta= \partial_1^{i_1}\cdot\ldots\cdot\partial_m^{i_m}$.
We  define {\bf order of derivative operator $\theta$}:
$$\ord(\theta) = i_1+\ldots+i_m\
\text {and}\
\Theta(s) =\{\theta\in\Theta| \ \ord(\theta) \le s\}.
$$ Let
$$R=\K\{y_j\:|\: 1\Le j\Le n\} := \K[\theta y_j\:|\: \theta\in\Theta, 1\Le j \Le n]\
$$
be a ring of commutative polynomials with coefficients in $\K$ in the 
infinite set of 
variables 
$\Theta Y=\Theta(y_j)_{j=1}^n$, and
$$R_s = \K\big[\Theta(s) y_j\big],\quad s\Ge 0.$$  
A ring $R$ is called a {\bf ring of differential polynomials
} in  differential indeterminate $y_1,\ldots, y_n$ 
over  $\K$.
\end{definition}

$R$  is differential ring with the set 
$\Delta = \{\partial_1,\ldots,\partial_m\}$.

\begin{definition}\label{3.2.38}
Let $\F$ be a differential field with a set of derivations
$\Delta = \{\partial_1,\ldots,\partial_m\}$.
The ring $D=\F[\partial_1,\dots,\partial_m]$ of skew polynomials in 
indeterminates $\partial_1,\dots,\partial_m$ with coefficients 
in $\F$ and the commutation rules 
$\partial_i \partial_j=\partial_j \partial_i,\ \partial_i a=
a\partial_i+ \partial_i(a)$  for all
$a\in \F,\ \partial_i,\partial_j \in \Delta$ is called 
a {\bf (linear) differential ($\Delta$-) operator ring}.
\end{definition}
If derivation operators are trivial on $\F$, then $D$ is
isomorphic to the commutative polynomial ring.

Below we consider the case  when $\K$ is the differential field $\F$
and char $\F=0$ only.  An ideal $I$ in $\F\{y_1,\ldots,y_n\}$ is 
called {\bf
  differential}, if $\partial
f\in I$ for all $f\in I$ and $\partial\in\Delta$.
We will denote by
$\{I\}$  the minimal perfect differential ideal containing $I$.
By the theorem (\cite{Kolchin}, p.126, Theorem 1), 
every perfect differential ideal can be represennted as an
intersection of finite number minimal
prime differential ideals: $\{I\}=\cap P_i$ 
(called the components of $\{I\}$).

Let
$\G=\F\left<\phi_1,\ldots,\phi_n\right>$ is finitely generated
 $\Delta$-extension of the differential field $\F$.
Define {\bf (Kolchin) differential
 dimension polynomial}
$\omega_{\phi_1, \ldots, \phi_n}(s)$
if
the following condition holds:
 $$\omega_{\phi_1, \ldots, \phi_n}(s)=
trdeg \F(\Theta(s)\phi_1,\ldots,\Theta(s)\phi_n)/\F$$
(here trdeg is the  transcendence degree field extension). 
In other words, we add
to the field $\F$ all derivations elements $\phi_1, \ldots,
\phi_k$ of order $1, \ldots, s$ 
and look for the cardinality of the 
maximum algebraically independent set of elements over $\F$
received field. 
For sufficiently large $s$ this dependence
is polynomial, or more precisely (see \cite{Kolchin}, p. 115, Theorem 7),
it is equal to the sum
$\sum_{j=1}^n\omega_{E_j}(s)$ of Kolchin polynomials of some subsets
$\N_0^m$.
Note that the
differential
dimension polynomial is the {\bf Hilbert polynomial }
filtered module of differentials of the extension $\G$ over $\F$.

This polynomial contains some $\Delta$-invariants of the field $G$ over
$F$ (in particular, the degree and leading coefficient), but
the polynomial can change when another system of generators 
is chosen
$\G=\F\left<\psi_1,\ldots,\psi_l\right>$. For example, polynomial
$\omega_{\F\left<\psi_1,\dots,\psi_n\right>}(s+1)=\omega_{\F\left<\Theta(1)\psi_1,\dots,\Theta(1 )\psi_n\right>}(s)$.

One of the invariants is introduced by W.Sit in \cite{Sit}
minimal dimension polynomial.

\begin{definition}\label{min_pol}(see \cite{Sit}).
The polynomial $\omega_{\eta_1,\dots,\eta_n/\F}(s)$ is called
{\bf minimal dimension differential polynomial }
for $\Delta$-extension $\G=\F\left<\eta_1,\dots, \eta_n\right>$ if for
 any system of $\Delta$-generators    
$\G=\F\left<\psi_1,\dots,\psi_k\right>$ 
we have
$\omega_{\psi_1,\dots,\psi_k/\F}(s)\geq\o_{\eta_1,
\dots,\eta_n/\F}(s)$ for all sufficiently large $s$.
In this case the polynomial $\omega_{\eta_1,\dots,\eta_n/\F}(s)$ 
will be
denoted by $\omega_{\G/\F}(s)$.
\end{definition}

W.Sit (\cite{Sit}, proposition 5)  proves the existence
$\omega_{\G/\F}(s)$.

\section{Basic results.}

\subsection{Macaulay constants of an numerical polynomial.
}

The article \cite{Kon} proves a criterion for the equality of 
a numerical polynomial  to
some  Kolchin polynomial $\omega_E(s)$. Denote by
$W$ is the set of all possible polynomials
$\{\omega_E(s): E\subset \N_0^m,m=1,\dots\}.$

\begin{theorem}\label{w}(\cite{KLMP}, proposition 2.4.10)
A numerical polynom $\omega(s)$ belongs to $W$, iff its 
sequence of minimizing coefficients consists only of non-negative integers.
\end{theorem} 

Note (see \cite{Kon}) that the set of Kolchin polynomials is 
closed under   additions:
\begin{theorem}\label{sum}(\cite{KLMP}, proposition 2.4.13.)
Let $\omega_1(s), \omega_2(s)\in W$. Then the polynomial 
$\omega(s)=\omega_1(s)+\omega_2(s)$ also belongs to $W$.
\end{theorem}
As follows from these theorems, the differential 
  dimension polynomial 
$\omega_{\eta_1,\dots,\eta_n/\F}(s)$
has non-negative minimizing coefficients.

\begin{theorem}\label{main}
Let $\omega(s)$ is a numerical polynomial of degree  $d$, and
$b(\omega)=(b_d,\dots ,b_0)\in\Z^{d+1}$ is a
sequence of its minimizing coefficients.
Then

\begin{equation}\label{Mac}
\omega(s)=\binom{s+d+1}{d+1}-\sum_{i=0}^{d+1}\binom{s+i-1-c_i}{i}
\end{equation}
where $$c_i=\sum_{j=i-1}^{j=d}b_i,\ i=0,\dots,d+1,$$
(the value of $c_{0}$ can be set arbitrarily).
\end{theorem} 

\begin{proof}
We will prove this theorem by induction on  $d$.
If $d=0$, $\omega(s)=b_0$, we  need to check the equality
$b_0=(s+1)-(\binom{s-1-c_0}{0}+\binom{s-c_1}{1})=(s+1)-(1+(s-c_1))$,
which follows from the definition $c_1=b_0$.

Let now $d>0$, $b(\omega)=(b_d,\dots ,b_0)\in\Z^{d+1}$.
According to the definition 
(\ref{min}), we denote
\begin{equation}\label{v}
v(s) = \omega(s+b_d) - \binom{s+d+1+b_d}{d+1} + \binom{s+d+1}{d+1}.
\end{equation}
Then $\deg v<d$,  $b(v)=(b_{d-1},\dots ,b_0)\in\Z^{d}$
and by induction we can assume that

\begin{equation}\label{v1}
v(s)=\binom{s+d}{d}-\sum_{i=0}^{d}\binom{s+i-1-c'_i}{i},
\end{equation}
where
$$c'_i=\sum_{j=i-1}^{j=d-1}b_i,\ i=0,\dots,d.$$
Let's make the substitution $s'=s+b_d$ in the expression (\ref{v}):
$$
v(s'-b_d) = \omega(s') - \binom{s'+1+d}{d+1} + \binom{s'-b_d+d+1}{d+1}.
$$
Then, taking into account the formula
(\ref{v1}) we will have:
\begin{align}  \notag
 \omega(s')=  \binom{s'+d+1}{d+1} - \binom{s'-b_d+d+1}{d+1}+ v(s'-b_d)= \\ \notag
\binom{s'+d+1}{d+1}  
- \binom{s'-b_d+d+1}{d+1}
+
\binom{s'-b_d+d}{d}-\sum_{i=0}^{d}\binom{s'-b_d+i-1-c'_i}{i}
=\\ \notag
\binom{s'+d+1}{d+1}
-
(\binom{s'-b_d+d+1}{d+1}-\binom{s'-b_d+d}{d})
-\sum_{i=0}^{d}\binom{s'-b_d+i-1-c'_i}{i}. \notag
\end{align}
From equality
$\binom{k+1}{l+1}-\binom{k}{l}=\binom{k}{l+1}$ we have
\begin{align} \notag
 \omega(s')=  \binom{s'+d+1}{d+1}-\binom{s'-b_d+d}{d+1}
-\sum_{i=0}^{d}\binom{s'-b_d+i-1-c'_i}{i}=\\   \notag
=  \binom{s'+d+1}{d+1}
-
\binom{s'+d-c_{d+1}}{d+1}
-\sum_{i=0}^{d}\binom{s'-b_d+i-1-c'_i}{i}=\\   \notag
\binom{s'+d+1}{d+1}
-\sum_{i=0}^{d+1}\binom{s'+i-1-c_i}{i},
\end{align}
т.к. $c'_i+b_d=c_i$ для  $i=0,\dots,d$, $c_{d+1}=b_d$.
\end{proof}

The values of $c_i$ in the expression (\ref{Mac}) are called 
{\bf Macaulay constants}
(see, for example, \cite{Dube}, formula (*), p. 768).
The arbitrary value $c_{0}$ for $w\in W$ can be considered 
equal to the smallest number,
for which the value of the Hilbert function is equal to the 
value of the polynomial.
For $E=(e_{ij})\subset\N_0^m,i=1,\dots,n$
this number equals to $\sum_{i=1}^m max_{j=1,\dots,n}e_{ij}$.

So, the theorem (\ref{main}) establishes a relationship between 
the values  of the minimizing coefficients
polynomial and its Macaulay constants.

From here and from the theorem (\ref{w}) we immediately have
\begin{corollary}
A numerical polynomial is a Kolchin dimension polynomial if and only 
if the sequence of its Macaulay constants is non-decreasing.
\end{corollary}

Later in this article, since we will not use the value of Macaulay's 
constant
$c_0$, we will consider only the values
$c_{d+1},\dots,c_1$ and, for convenience, number them, 
as well as the minimizing coefficients, starting from zero:
$(c_d,\dots,c_0)$, where $d$  is polynomial degree.
\begin{example}
Consider the following system of linear differential equations:
\begin{eqnarray*}
\partial_1^2\partial_2\xi=0 \\
\partial_1\partial_2^2\xi=0 \\
\partial_1\partial_2\partial_3^2\xi=0 \\
\dots\\
\partial_1\partial_2\partial_3\dots \partial_m^2\xi=0 
\end{eqnarray*}
Let's find the dimension polynomial of this system and its Macaulay constants.

Since the equations are linear and form a Gr\"obner basis in the ring 
of differential operators,
we need to calculate the dimension polynomial of the lower 
triangular matrix $E\subset \N_0^m$,
with 2 on the diagonal, and 1 below it:
$$
\begin{matrix}
2 0 0\dots 0\\
1 2 0\dots0 \\
\dots\\
1 1\dots 1 2
\end{matrix}
$$
To calculate a dimension polynomial
we will apply induction and the formula
polynomial changes when adding an element, (see \cite{KLMP}, theorem 2.2.10).

Let $m=2$, $e=(1, 0)$, 
\begin{equation}\label{formula}
\omega_E(s)=\omega_{E\cup e}(s)+\omega_H(s-1),
\end{equation}
where $H$ is a matrix 
subtracted from each row $E$ of vector $e$,
$\ord e=1$. We have: $\omega_E(s)=\omega_e(s)+
\omega_{
\begin{smallmatrix}
10\\01
\end{smallmatrix}
}
(s-1)=(s+1)+1$.      
In particular, the sequence of minimizing coefficients    
equals
$b(\omega_E)=(1,1)$.

Now suppose $m>2$.
Let $e=(1,0,\dots,0)$ and apply the same
formula:
$\omega_E(s)=\omega_e(s)+
\omega_{ H}(s-1)$, where $H$ i a matrix:
$$
\begin{matrix}
1 0 0\dots 0\\
0 2 0\dots0 \\
\dots\\
0 1\dots 1 2
\end{matrix}
$$      

It is clear that the dimension polynomial of the matrix $H$
is equal to the dimension polynomial of the matrix 
$E'\subset\N_0^{m-1}$,
for which, by the inductive hypothesis, holds
$b(\omega_{E'})=(1,\dots,1)\subset\N_0^{m-1}$.

By definition of minimizing coefficients (\ref{min}),
from the condition $\omega_e(s)=\binom{s+m-1}{m-1}$ we get
$b(\omega_{E})=(1,\dots,1)\subset\N_0^{m}$.
From the theorem (\ref{main}), the Macaulay constants are:
$(1,2,\dots,m)$.

Let us now calculate the standard coefficients of the dimension polynomial.
Let for $E\subset\N_0^m$ the standard coefficients be the numbers
$a(\omega_{E})=(a_{m-1},\dots,a_0)$ which means representation
$$\omega_E(s)=\sum_{i=0}^{m-1}a_i\binom{s+i}{i}.$$
Denote the dimension polynomial of the matrix $m\times m$ by
$\omega_m(s)$.
As shown above, for $m=2$ the dimension polynomial
equals $(s+1)+1$, so it's
sequence of standard coefficients
 is $(1,1)$.
Let $m>2$.
Denote by $\nabla$ the difference operator acting
on polynomials: $\nabla f(s)=f(s)-f(s-1)$.
Taking into account the formula
$\nabla\binom{s+i}{i}=\binom{s+i-1}{i-1}$ and theorem (\ref{main})
about the representation of a polynomial in terms of its Macaulay constants, we see that
$\nabla\omega_m(s)$ is a polynomial whose Macaulay constants
are equal to $c_{m-1},\dots,c_1$ and we can use the inductive assumption
to calculate standard coefficients $\omega_m(s)$, $a_i(\omega_m)=a_{i-1}(\omega_{m-1})$, $i=m-1,\dots,1$.
It remains for us to calculate the lowest standard coefficient, $a_0(\omega_m)$. 
We use the formula (\ref{formula})
$\omega_m(s)=\binom{s+m-1}{m-1} +\omega_{m-1}(s-1)$.
Substitute in this formula $s=-1$, from the equality
$\binom{s+i}{i}|_{s=-1}=0$ for $i>0$
we get $a_0(\omega_m)=\sum_{i=0}^{m-2}a_i(\omega_{m-1})\binom{-1+i}{i}$.

Compute
$\binom{s+i-1}{i}|_{i=-1}
=\binom{i-2}{i}=
\frac{(i-2)(i-3)\dots}{i!}$,
where there are exactly $i$ factors in the product from above. It's clear that
if $i>1$, this product is equal to zero. We have
$a_0(\omega_m)=
a_1(\omega_{m-1})(1-2)+a_0(\omega_{m-1})=-a_2(\omega_m)+a_1(\omega_m)$.

We have a recursive formula, and can calculate all the standard coefficients of the polynomial $\omega_m$:
for $m=3$ it is (1,1,0), for $m=4$ it is (1,1,0,-1), then
(1,1,0,-1,-1,0,1,\dots,).
\end{example}

\subsection{
Macaulay constants of a minimal dimension polynomial.}

According to Sit's theorem (\cite{Sit}, Proposition 5),
set $W$ is
well-ordered with respect to the above 
(definition \ref{min_pol}) order.
Note that this order is equivalent to the lexicographic order
on a sequence of minimizing coefficients (from highest to lowest)
and on a sequence of Macaulay constants.
If you need to compare sequences of different lengths,
the shorter sequence must be left-completed with zeros.

The simplest example of a differential field extension, for which
succeed
specify the minimal dimension polynomial  is an extension,
given by one equation.
\begin{theorem}\label{first}
Let $F\in \F\{y_1\}$.
Then for any prime differential component $P$
of ideal $\{F\}$
holds
\begin{equation}\label{min1}
\omega _P(s)=\binom{s+m}{m}-\binom{s+m-d_j}{m},
\end{equation}
and the  sequence of Macaulay constants of the minimal 
dimension polynomial 
for each component is constant
(may be different for different prime components).
Conversly, if the Macaulay constants of the differential
dimension polynomial of the field extension 
$\omega_{\xi_1,\dots,\xi_n/\F}$
are the same and $\deg (\omega_{\xi_1,\dots,\xi_n/\F})=|\Delta|-1$, 
then this polynomial is minimal $\omega_{\G/\F}$,
and a prime ideal, defining the field $\G$,
is the general component one differential polynomial 
$F\in \F\{y_1\}$.
\end{theorem}
\begin{proof}
By the components theorem (\cite{Kolchin}, p.185, theorem 5), 
every minimal prime differential
ideal containing $\{F\}$, may be
represented as $P_j=[G_j]:H_{G_j}^\infty$, 
where $G_j\in\F\{y_1\}$ is an
irreducible differential polynomial.
Denote  $d_j=\ord G_j$ the order of this polynomial. We obtain
$$\omega_{\xi_1}(s)=
\binom{s+m}{m}-\binom{s+m-d_j}{m},$$
where $\xi_1$ is the common zero of $P_j$, therefore
$\omega_{P_j}(s)\leq \binom{s+m}{m}-\binom{s+m-d_j}{m}$.
Sequence of minimizing coefficients
polynomial (\ref{min1}) is
$(d,0,\dots,0)$,  so
inequality $\omega_{P_j}>=\binom{s+m}{m}-\binom{s+m-d_j}{m}$
will follow from the differential type and typical $\Delta-$ dimensions
invariance theorem
(see, for example, \cite{KLMP}, collarary 5.4.7). 

According to the theorem (\ref{main})
the Macaulay constants of the polynomial (\ref{first}) are equal to $(d_j,\dots,d_j)$.

Conversely, let all the Macaulay constants of the minimum dimension
polynomial $\omega_{\G/\F}$ be the same and equal to $d$, 
differential type   $\G$ over $\F$ is equal to $m-1$ 
(this condition means that the polynomial
$\omega_{\G/\F}$ has degree $m-1$, see \cite{Kolchin}, p. 118) and
$\xi_1,\dots,\xi_n$ -- common zero,
for which $\omega_{\xi_1,\dots,\xi_n/\F}(s) =\omega_{\G/\F}(s)
=\binom{s+m}{m}-\binom{s+m-d}{m}$.

By Theorem (\cite{Kolchin}, p. 115, Theorem 7)
$\omega_{\xi_1,\dots,\xi_n/\F}=\sum_{j=1}^n\omega_{E_j}$, and 
since all the minimizing coefficients of $\omega_{E_j}$ polynomials 
are non-negative, from the condition
$\omega_{\xi_1,\dots,\xi_n/\F}=\binom{s+m}{m}-\binom{s+m-d}{m}$
it follows that $\omega_{E_j}=0$ for all $j$, except for the total 
one (we can assume that for all
$j>1$). Because $\omega_{E_j}=0$ for $j>1$, the elements $\xi_2,\dots,\xi_n$ are algebraic over $\F\left<\xi_1\right>$ and $\F\left<\xi_1 ,\dots,\xi_n\right>=\F\left<\xi_1\right>$. 
Since
$E_1$ consists of one element, a prime ideal defining
$\F\{y_1\}$, is the general component of some differential 
polynomial.
\end{proof}

If degree
$\omega_{\psi_1,\ldots,\psi_n/\F}(s)$ is less than $m$, and the field $\F$ contains
$\C(x_1, \ldots, x_m)$ (field of rational functions in unknowns
$x_1, \ldots, x_m$), then (see \cite{Kolchin}, part II, \$8, assertion 9)
$\G$ has a
primitive element $\xi$ (i.e. such that
$\G=\F\left<\xi\right>$).
Moreover, $\xi$ can be chosen in the form
$\xi=\sum_{j=1}^n\lambda_j\psi_j$, $\lambda_j\in\F,
j=1,\ldots,n$. With such a changing, the differential dimension
polynomial does not increase, i.e. $\omega_{\xi/\F}(s)\le \omega_{\psi_1,
\ldots,\psi_n/\F}(s)$ for all sufficiently large $s$ and so
the problem of finding the minimal differential dimension
polynomial is related to the search for a primitive element
fields (although it doesn't solve it, see \cite{KLMP}, Example 5.7.7).
Both  the problem of calculating the polynomial 
$\omega_{\psi_1,\ldots,\psi_n/\F}(s)$,
and the searching  $\xi$ are "in principle" solved 
using the algorithm
Rosenfeld-Grebner, included in the diffag package of the system
Computer Algebra Maple.

Suppose a primitive element is found in the extension,
$b(\omega_{\xi/\F})=(b_d,b_{d-1},\dots.b_0)$ and $b_i\ne0$
for some $i<d$. We are interested in whether there are such values
$b_i$ that $b(\omega_\G/\F)(s))=(b_d,0,\dots,0)$
is impossible.

The following example gives a negative answer to this question.
\begin{example}\label{ex2}
Let $\F=\C(x_1,x_2)$, $k>=2$ and $\Delta$-extension $\G$ be given by the system
$$
  \begin{cases} \partial_1\partial_2^2\varphi=0, &
        \\x_1 x_2 \partial_1^k \partial_2\varphi-x_1\partial_1^k\varphi+x_2\partial_1^{k-1}\partial_2\varphi+\partial_1\partial_2\varphi-\partial_1^{k-1}\varphi=0.&\end{cases}
$$

These equations form the Gr\"obner basis, so
$\omega_{\varphi/\F}(t)=
\omega_{\left(\begin{smallmatrix} 1 & 2\\ k & 1\end{smallmatrix}\right)}(s)=2s+k$ and $b(\omega_{\varphi/\F}) =(2,k-1)$. Let
$\psi=\varphi-x_1 x_2(\partial_1^{k-1}\varphi-x_2\partial_1^{k-1}\partial_2\varphi)$. This change of variables is invertible, 
since $\varphi=(1+x_1x_2\partial_1^{k-1})\psi$.
At the same time $\partial_1\partial_2\psi=0$, hence
$\omega_{\psi/\F}(s)=\omega_{\G/\F}(s)=2s+1$ and $b(\omega_{\G/\F})=(2,0 )$.
\end{example}

In the example (\ref{ex2}), the field contains elements that are not constants. This condition, as the following statement shows,
is essential.
\begin{theorem}
Let $\G$ be defined by the system $\Sigma$ of linear differential 
equations in one variable  with coefficients,
which are constants of the field $\F$. If
$\omega_{\xi/\F}(s)!=\binom{s+m}{m}-\binom{s+m-d}{m}$, then
$\omega_{\G/\F}(s)!=\binom{s+m}{m}-\binom{s+m-d}{m}.$
\end{theorem}
\begin{proof}
Consider $\Omega_{\G/\F}$, the module of differentials (see, for example,
\cite{KLMP}, p.38), and let $J=[\Sigma]$ be the ideal of the ring $D=\G[\Delta]$ of differential
operators. We have the exact sequence of $D$-modules:
$$
0\to _{D}J\to D \to\Omega_{\G/\F}\to 0.
$$

Suppose $\omega_{\G/\F}(s)=\binom{s+m}{m}-\binom{s+m-d}{m}$.
By Theorem 5.7.8 (\cite{KLMP}), there exists
this exact sequence:
$$
0\to _{D}D\to D \to\Omega_{\G/\F}\to 0.
$$

It follows from these representations, that $_{D}J\oplus D=D\oplus D$.
Hence the ideal $J$ is a projective $D$-module.
Under the conditions of the theorem, we can assume that the ring $D$ 
is a ring of commutative polynomials over a field, and then
$J$ must be free (principal ideal in a polynomial ring).
Then, by the  theorem (\ref{first}), holds
$\omega_{\xi/\F}=\binom{s+m}{m}-\binom{s+m-d}{m}$.
This is contrary to the condition.
\end{proof}
As follows from the proof, the example (\ref{ex2}) represents the projective ideal
in non-commutative ring $D=\C(x_1,x_2)[\partial_1,\partial_2]$, which is not free.

\end{document}